\documentclass[12pt]{amsart}
\usepackage{amssymb}
\usepackage{mathrsfs}
\usepackage{amssymb}
\usepackage{amsfonts}
\usepackage[english]{babel}
\usepackage[T1]{fontenc}
\usepackage[latin1]{inputenc}
\usepackage{fullpage}
\usepackage{amsmath}
\usepackage[all]{xy}
\usepackage{stmaryrd}

\newtheorem{theorem}{Theorem}[section]
\newtheorem{lemma}[theorem]{Lemma}
\newtheorem{corollary}[theorem]{Corollary}
\theoremstyle{definition}
\newtheorem{definition}[theorem]{Definition}

\newtheorem{proposition}[theorem]{Proposition}

\theoremstyle{remark}
\newtheorem{remark}[theorem]{Remark}

\numberwithin{equation}{section}

\begin{document}

\title{ Brown-York mass and positive scalar curvature II\\
- Besse's conjecture and related problems}


\author{Yi Fang}
\address{(Yi Fang) Department of Mathematics, Anhui University of Technology, Ma'anshan, Anhui 243002, China}
\email{flxy85@163.com}

\author{Wei Yuan}
\address{(Wei Yuan) Department of Mathematics, Sun Yat-sen University, Guangzhou, Guangdong 510275, China}
\email{gnr-x@163.com}





\thanks{Yi Fang was supported by \emph{The Young Teachers' Science Research Funds of Anhui University of Technology} (Grant No. RD16100248 ). Wei Yuan was supported by NSFC (Grant No. 11601531, No. 11521101) and \emph{The Fundamental Research Funds for the Central Universities} (Grant No. 2016-34000-31610258).}

\begin{abstract}
The Besse's conjecture was posted on the well-known book \emph{Einstein manifolds} by Arthur L. Besse, which describes the critical point of Hilbert-Einstein functional with constraint of unit volume and constant scalar curvature. In this article, we show that there is an interesting connection between Besse's conjecture and positive mass theorem for Brown-York mass. With the aid of positive mass theorem, we investigate the geometric structure of the so-called CPE manifolds, which provides us further understanding of Besse's conjecture. As a related topic, we also have a discussion of corresponding results for $V$-static metrics. 
\end{abstract}

\maketitle



\section{Introduction}

In Riemannian geometry, scalar curvature is one of the most mysterious object. It is defined in such a simple way but surprises one with its great geometric and topological influences on manifolds. \\

As one type of the curvature closed related to Ricci curvature, scalar curvature appears in the study of Einstein metrics frequently. One can easily find many such topics on the classic book \emph{Einstein manifolds} by Arthur L. Besse. Besides researches on Einstein manifolds, it also includes many open problems about Einstein metrics. For example, it proposed an open question as follow in Remark 4.48 (page 128 in \cite{Besse}):\\

\emph{Suppose $M^n$ is an $n$-dimensional compact manifold without boundary. Let
$$
\mathcal{C}=\{g:\text{Vol}_{M^n}(g) =1 \ \text{and} \ R(g)=\text{constant}\}
$$
be a collection of Riemannian metrics on $M$. Consider the Hilbert-Einstein functional
\begin{equation*}
\mathcal{F}(g)=\int_{M}R_g dv_g
\end{equation*}
with constraint $g\in \mathcal{C}$. It was conjectured that critical points in this case have to be Einstein.}\\

It was pointed out on exactly the same page in Proposition 4.47 in \cite{Besse}, this conjecture holds if $R_g/(n - 1) \notin Spec( - \Delta_g)$. This leaves the case open only when $R_g/(n - 1)$ is an eigenvalue of Laplacian.\\

On the other hand, even if $R_g/(n - 1) \in Spec(-\Delta_g)$, the Euler-Lagrange equation can still be written as
\begin{equation}\label{eqn:CPE}
\gamma_g^* u :=\nabla^2 u -g\Delta_g u- u Ric_g = E_g,
\end{equation}
where $\gamma_g^*$ is the $L^2$-formal adjoint of the linearization of scalar curvature and $u$ is a smooth function with zero mean and $E_g := Ric_g - \frac{R_g}{n} g$ is the traceless part of Ricci curvature tensor. Equation (\ref{eqn:CPE}) is usually referred to be \emph{critical point equation (CPE)} in the study of Besse's conjecture. Easy to see that the conjecture would be confirmed if the function $u$ vanishes identically. However, the existence of non-trivial functions solving equation (\ref{eqn:CPE}) would make the situation mysterious. In fact, if the conjecture still holds in this case, the manifold has to be isometric to the canonical sphere $\mathbb{S}^n$ due to the well-known \emph{Obata's theorem}, if we normalized the scalar curvature to be $n(n - 1)$. Thus the Besse's conjecture is also referred to be the following one:\\

\newtheorem*{conj_A}{\bf Conjecture A}
\begin{conj_A}
Suppose $(M^n,g)$ is a closed Riemannian manifold with constant scalar curvature
$R_g =n(n-1)$.
Assume there is a non-zero smooth function $u$ solves the equation
\begin{equation*}
 \tag{CPE}  \nabla^2 u -g\Delta_g u - u Ric_g = E_g,
\end{equation*}
then $(M^n,g)$ is isometric to the round sphere $\mathbb{S}^n$ with its canonical metric.
\end{conj_A}
For simplicity, we say a metric $g$ is a \emph{CPE metric}, if it satisfies CPE equation for some smooth function $u$. And we also use a triple $(M,g,u)$ stands for this CPE manifold.\\

Even before the publication of Besse's book, partial results were known for this conjecture. For example, Lafontaine showed this conjecture is true, if the CPE metric is also conformally flat (see \cite{Lafontaine}). Among all those partial results about Besse's conjecture, Hwang and his collaborators did a series of studies on this special topic over years (see \cite{C-H,C-H-Y_1,C-H-Y_2,C-H-Y_3,Hwang_1,Hwang_2,Hwang_3}). Their researches provide us many fundamental interpretations on the geometric structures of CPE metrics and related geometry. Hwang first observed in \cite{Hwang_1} that the level set
$$\Sigma:=\{x \in M : u(x)=-1 \}$$
plays an essential role in studying the geometry of CPE metrics (for this reason, we will refer $\Sigma$ to be the \emph{critical level set}). Moreover, from a very delicate analysis on the geometric structure of $\Sigma$ and the assistance of some geometric equalities, Chang-Hwang-Yun verified Besse's conjecture with the additional assumption that $g$ is of harmonic curvature (see \cite{C-H-Y_3}). This result successfully generalizes Lafontaine's work and made a large progress in solving Besse's conjecture.\\

By noticing the analogue between CPE metrics and \emph{vacuum static metrics}, which satisfy the equation
\begin{equation}\label{eqn:vacuum_static}
\gamma_g^*u =\nabla^2 u - g\Delta u - u Ric_g = 0
\end{equation}
for some non-trivial function $u$, Qing and Yuan showed a CPE metric has to be spherical, if it is Bach flat. (In fact, this result can be improved for three dimensional manifolds. For more details, please see \cite{Qing-Yuan}; also see \cite{Barros-Ribeiro} for a four dimensional result). \\

The main purpose of this article is to introduce a new perspective of studying Besse's conjecture. Specifically speaking, we find an interesting connection between Besse's conjecture and the \emph{positive mass theorem for Brown-York mass} due to Shi-Tam (\cite{Shi-Tam}), which is equivalent to the well-known \emph{positive mass theorem for ADM mass} by Schoen-Yau and Witten (\cite{Schoen-Yau_1, Schoen-Yau_2, Schoen-Yau_3, Witten}). This idea did not appear in the previous research on Besse's conjecture according to the best of the authors' knowledge, although it seems natural since both of them origins from the Hilbert-Einstein functional. \\

The main tools we used in this article is the positive mass theorem for Brown-York mass (\cite{Shi-Tam}) and a conformal transformation trick first showed up in Bunting and Masood-ul-Alam's well-known work of the uniqueness of static black-holes (\cite{Bunting-Masood}). The same method can be applied to the study of first eigenvalue of Laplacian for manifolds with boundary and positive scalar curvature, please refer to the first article in the sequence (\cite{Yuan_2}). Similar approaches were also used in other works such as \cite{H-M-R, Qing}. \\

In order to derive higher dimensional results, we need to use the corresponding version of Shi-Tam's positive mass theorem. It is equivalent to the positive mass theorem for ADM mass in higher dimensions, which was claimed by Lohkamp and Schoen-Yau independently (see \cite{Lohkamp_1, Lohkamp_2, Schoen-Yau_3}). \\

The main result of this article is the following:

\newtheorem*{thm_A}{\bf Theorem A}
\begin{thm_A}
	Let $(M^n,g, u)$ be a CPE manifold with scalar curvature $R_g =n(n-1)$. Suppose its critical level set $\Sigma=\bigcup^m_{i=1}\Sigma_i$ and each connected component $(\Sigma_i, g|_{\Sigma_i})$ can be isometrically embedded into $\mathbb{R}^n$ as a convex hypersurface with mean curvature $A^i_0$,
	then we have inequalities	
	\begin{align}
	Area (\Sigma_i, g)\leq \frac{ (\alpha + 1)^2 W_i^2 }{(n-1)(n-2) (W_i^2 + 1 + \alpha)^2 } \int_{\Sigma} \left( R_{\Sigma_i} + |\overset{\circ}{A}{}^i_0|^2 \right) d\sigma_g,
	\end{align}
	where $\overset{\circ}{A}{}^i_0$ is the traceless part of $A^i_0$, $W_i :=|\nabla u|_{\Sigma_i}$,  and
	$$\alpha: = \left( u^2 + |\nabla u|^2 \right)^{\frac{1}{2}}.$$
	Moreover, $(M^n, g)$ is isometric to the round sphere $\mathbb{S}^n$, if one of the equalities holds.
\end{thm_A}

\begin{remark}
	According to the solution of Weyl problem, the isometrical embedding assumptions can be replaced by assumptions on secional curvatures, for instance Gaussian curvature of $\Sigma$ is positive when $n=3$. For the purpose of maintaining the simplicity of assamptions, we state our results in the current way throughout the whole article instead of involving more complicated assumptions.  
\end{remark}
\ \\

Adopting a similar approach of Shen (cf. \cite{Shen}), one can derive some topological informations of $\Sigma$ for three dimensional CPE manifolds without embeddedness assumptions on $\Sigma$. For higher dimensions, this approach can provide improved estimate, if $\Sigma$ is connected: 

\newtheorem*{thm_B}{\bf Theorem B}
\begin{thm_B}\label{thm:CPE_est_sigma_connected}
	Let $(M^n,g, u)$ be an $n$-dimensional CPE manifold with scalar curvature $R_g=n(n-1)$. Suppose its critical level set $\Sigma$ is connected, then
	\begin{align}
	Area(\Sigma, g) 
	\leq& \frac{W^2}{(n-1)(n-2)(1+W^2)} \int_{\Sigma} R_{\Sigma} d\sigma_g,
	\end{align}	
	where $W = |\nabla u|_\Sigma$ and equality holds if and only if $(M^n , g)$ is isometric to the round sphere $\mathbb{S}^n$ with its canonical metric. 
	
	For $n=3$, there is at least one component of $\Sigma$ is homeomorphic to $\mathbb{S}^2$. Furthermore if $\Sigma$ is connected, we have the inequality	
	\begin{equation}
	\text{Area}(\Sigma, g)\leq \frac{4 \pi W^2}{ 1 + W^2},
	\end{equation}
	where $(M^n, g)$ is isometric to the round sphere $\mathbb{S}^3$, if and only if the equality holds.
\end{thm_B}

\begin{remark}\label{rem:B-D-R-R_remark}
	Based on the same idea, Batisa-Di\'{o}genes-Ranier-Ribeiro get similar conclusions for the boundary of three dimensional $V$-static manifolds (see \cite{B-D-R-R}). As we will state below, these two results are in fact equivalent to each other. However, since these results never showed up in the research of Besse's conjecture so far as we know, we believe it is still worth to state it here as a special case of our general results.
\end{remark}
\begin{remark}	
	 So far we could not rule out the possibility that $\Sigma$ might have components more than two. It would be a very interesting question whether we can prove the connectedness of $\Sigma$ without any additional assumption. This would be a great progress in solving Besse's conjecture. 
\end{remark}
\ \\

Besides CPE metrics, there is another type of metrics closely related to the vacuum static metrics, called \emph{$V$-static metrics}. It was first introduced by Miao and Tam to study the critical metrics of volume functional with constraint of constant scalar curvature (\cite{Miao-Tam_1}). To be precise, these metrics are defined to be those satisfies the equation
\begin{equation}\label{eqn:V-static}
\gamma_g^*f =\nabla^2 f - g\Delta f - f Ric_g = \mu g
\end{equation} 
for some smooth function $f$ and $\mu \in \mathbb{R} - \{0\}$. Here we adopt the definition in \cite{C-E-M} instead of the original one in \cite{Miao-Tam_1} (where $\mu$ is taken to be $1$) and explain the reason later. \\

The concept of \emph{V-static metrics} in fact possesses a very strong geometric and variational meaning:
Corvino-Eichmair-Miao discovered the volume-scalar curvature stability of non-$V$-static metrics (\cite{C-E-M}). On the contrary, the second named author obtained a local volume comparison for $V$-static metrics, which showed that the deformation result of Corvino-Eichmair-Miao is in fact sharp (see \cite{Yuan_1}).  \\

Similar to Besse's conjecture, mathematicians are also interested in classifying all possible $V$-static metrics. Towards this goal, Miao and Tam classified such metrics with additional curvature assumptions that either it is Einstein or locally conformally flat (\cite{Miao-Tam_2}). There are many other partial results, please see \cite{B-D-R-R} for more references.\\

Note that if we simply taking $\mu = - \frac{R_g}{n}$ and $f = u + 1$, we obtain the CPE equation (\ref{eqn:CPE}) immediately from $V$-static equation (\ref{eqn:V-static}). Similarly, we can derive $V$-static equation (\ref{eqn:V-static}) from CPE equation (\ref{eqn:CPE}) by choosing $u = - \frac{R_g}{n\mu}f - 1$ for $\mu \neq 0$. (For $\mu = 0$, the volume functional behaves very differently. Please refer to the remark 1.3 in \cite{Yuan_1} for more details.) In this sense, CPE metrics and $V$-static metrics are in fact equivalent to each other. This fact was first pointed out by Corvino-Eichmair-Miao in the beginning of \cite{C-E-M}.\\  

Of course, the usual setting for $V$-static metrics are manifolds with boundary. This is more extensive than the original setting of Besse's conjecture. Besides this, we also would like to include the case when $\mu = 0$ to make the result more complete. Due to these reasons, we include the parameter $\mu$ in our setting. Now we state the $V$-static metric version of Theorem A and note that this is in fact equivalent to Theorem A when $\mu > 0$: 

\newtheorem*{cor_A}{\bf Corollary A}
\begin{cor_A}
	Let $(M^n,g, \mu, f)$ be a $V$-static manifold with $\mu \geq 0$ and scalar curvature $R=n(n-1)$. Suppose its boundary $\Sigma=\bigcup^m_{i=1}\Sigma_i$ and each connected component $(\Sigma_i, g|_{\Sigma_i})$ can be isometrically embedded into $\mathbb{R}^n$ as a convex hypersurface with mean curvature $A^i_0$,
	then we have inequalities	
	\begin{align}
	Area (\Sigma_i, g)\leq \frac{ (n-1)( \mu + (n-1)\alpha)^2 W_i^2 }{(n-2) ( \mu^2 + (n-1) \mu \alpha + (n-1)^2W_i^2)^2 } \int_{\Sigma} \left( R_{\Sigma_i} + |\overset{\circ}{A}{}^i_0|^2 \right) d\sigma_g,
	\end{align}
	where $\overset{\circ}{A}{}^i_0$ is the traceless part of $A^i_0$, $W_i :=|\nabla f|_{\Sigma_i}$,  and
	$$\alpha: = \left( f^2 + |\nabla f|^2 \right)^{\frac{1}{2}}.$$
	Moreover, $(M^n, g)$ is isometric to a geodesic ball in the round sphere $\mathbb{S}^n$, if one of the equalities holds.
\end{cor_A}

\begin{remark}
	When $\mu = 0$, the $V$-static equation reduces to the vacuum static equation as we described before. In this case, $\Sigma$ is usually referred to be an \emph{event horizon} and the corresponding area estimate was obtained by the second named author in \cite{Yuan_2}.
\end{remark}

\begin{remark}
	For $n = 3$ and $\Sigma$ is connected and totally umbilical in $\mathbb{R}^n$, this estimate recovers the one by Batisa-Di\'{o}genes-Ranier-Ribeiro in \cite{B-D-R-R} as we stated in Remark \ref{rem:B-D-R-R_remark}. 
\end{remark}

\paragraph{\textbf{Acknowledgement}}
The authors would like to express their appreciations to Professor Qing Jie for his constantly support and encouragement. We also would like to thank Professor Li Tongzhu for inspiring discussions back in University of California, Santa Cruz.

\ \\

\section{A sharp inequality for $\mathscr{L}_n$-extension}

Let $(M^n, g)$, $n\geq3$, be an $n$-dimensional compact Riemannian manifold with boundary $\partial M:=\Sigma$. Considering the Schr\"odinger operator
\begin{equation}
\mathscr{L}_\lambda := - \Delta_g - \lambda : C^\infty(M) \rightarrow C^\infty(M),
\end{equation}
where $\lambda \in \mathbb{R}$. \\

Analogous to the notion of harmonic extension, we introduce a similar one with respect to the operator $\mathscr{L}_\lambda$; in particular, it coincides with harmonic extension when $\lambda = 0$.\\

 \begin{definition}\label{def:L_lambda_extension}
Let $w\in C^\infty(\Sigma)$ be a smooth function on the boundary $\Sigma$, we say $\varphi$ is an $\mathscr{L}_\lambda$-extension of $w$, if
  \begin{equation}
    \left\{
\begin{split}
\mathscr{L}_\lambda \varphi&=0,   \quad &\text{on} \  M\\
 \varphi&=w, \quad &\text{on} \  \Sigma.
 \end{split}
\right.
\end{equation}
In particular for $w \geq 0$, we say $\varphi$ is a \emph{positive $\mathscr{L}_\lambda$-extension} of $w$, if $\varphi > 0$ on $M \backslash \Sigma$.
 \end{definition}

\begin{remark}
The extension $\varphi$ always exists, if $\lambda\notin\text{Spec}(-\Delta_g)$. Moreover, if $w \not\equiv 0$ on $\Sigma$ additionally, then $\varphi$ is an non-trivial extension. If $\lambda \in\text{Spec}(-\Delta_g)$ and $w$ vanishes identically on $\Sigma$, then $\varphi$ is an eigenfunction of $(- \Delta_g)$ and hence can be chosen to be a positive extension of $w = 0$.
\end{remark}

 For our purpose of this article, we only consider positive $\mathscr{L}_n$-extensions, say $\varphi$, for a given non-negative smooth function $w$ defined on the boundary $\Sigma$. That is, $\varphi$ solves the equation
\begin{equation}
    \left\{
\begin{split}
\Delta_g \varphi + n \varphi&=0,   \quad &\text{on} \  M \\
 \varphi&=w, \quad &\text{on} \  \Sigma
 \end{split}
\right.
\end{equation}
with $\varphi > 0$ on $M \backslash \Sigma$. For simplicity, we denote
\begin{equation}\label{alpha}
\alpha:=\max_{M}\left(\varphi^2+|\nabla\varphi|^2 \right)^{\frac{1}{2}} > 0
\end{equation}
and $\rho:= \left(1 + \alpha^{-1} w \right)^{-1}$.

\begin{definition}
Let $(M^n,g)$ be a compact Riemannian manifold with boundary $\Sigma = \bigcup_{i=1}^m \Sigma_i$. For a smooth non-negative function $w \in C^\infty(\Sigma)$, we say $(\Sigma, g|_\Sigma)$ is \textbf{$\mathscr{L}_n^w$-regular}, if each connected component $\Sigma_i$ satisfies
\begin{itemize}
\item $(\Sigma_i, \rho^2 g|_{\Sigma_i})$ can be embedded in $\mathbb{R}^n$ isometrically as a convex hypersurface;
\item the mean curvature of $\Sigma_i$ with respect to $g$ satisfies $$H_g^i > (n-1)(\alpha+w_i)^{-1}\partial_{\nu_g}\varphi ,$$ where $w_i = w|_{\Sigma_i}$ and $\varphi$ is a positive $\mathscr{L}_n$-extension of $w$.
\end{itemize}
\end{definition}

An important property of the $\mathscr{L}_n$-extension of a given smooth function defined on the boundary is the following geometric-analytical inequalities hold on the manifold, if appropriate geometric assumptions are posed on the manifold:
\begin{theorem}\label{thm:L_n-inequality}
Let $(M^n,g)$ be an $n$-dimensional compact Riemannian manifold with boundary $\Sigma=\bigcup^m_{i=1}\Sigma_i $ and scalar curvature $$R_g\geq n(n-1).$$ 

Consider a smooth non-negative function $w \in C^\infty(\Sigma)$ with its positive $\mathscr{L}_n$-extension $\varphi$. Suppose the boundary $(\Sigma, g|_\Sigma)$ is $\mathscr{L}_n^w$-regular, then following inequalities hold
 \begin{equation}
 \int_{\Sigma_i}H_g^i \left(1+\frac{w_i}{\alpha} \right)^{2-n}d\sigma_g \leq \int_{\Sigma_i}\left( \hat H_0^i +(n-1)\frac{\partial_\nu \varphi}{\alpha} \right) \left(1+\frac{w_i}{\alpha} \right)^{1-n}d\sigma_g, \quad i=1,\cdots, m,
 \end{equation}
where $\hat H_0^i$ is the mean curvature of $(\Sigma_i, \rho^2 g|_{\Sigma_i})$ embedded in $\mathbb{R}^n$ with 
$$\alpha:=\max_{M}\left(\varphi^2+|\nabla\varphi|^2 \right)^{\frac{1}{2}}$$
and
$$\rho:= \left(1 + \alpha^{-1} w \right)^{-1}.$$
 Moreover, the equality holds for some component $\Sigma_{i_0}$ if and only if $(M^n,g)$ is isometric to a convex domain in the round sphere $\mathbb{S}^n$ and $\varphi$ is the restriction of a first eigenfunction of $\Delta_{\mathbb{S}^n}$ on it.
\end{theorem}

\begin{remark}
When $w = 0$ and $\lambda$ is less than or equals to the first eigenvalue of Laplacian, this inequality recovers the one in \cite{Yuan_2}. It would be interesting to compare differences between both of the conclusion and proofs.
\end{remark}

Before giving the proof of those desired inequalities, we get some preparations first. \\

Consider the conformal metric $\hat{g}:=u^{\frac{4}{n-2}}g$, where $u=(1+\alpha^{-1} \varphi)^{-\frac{n-2}{2}}$. Note that we have $2^{- \frac{n-2}{2}} \leq u \leq 1$, since $\varphi$ is a positive extension.

\begin{lemma}\label{lem:conformal_scalar_curvature}
Suppose the scalar curvature satisfies 
$$R_g \geq n(n-1),$$
then the scalar curvature of the conformal metric $\hat g$ is non-negative and it vanishes identically if and only if $R_g = n(n-1)$ and $\varphi^2+|\nabla\varphi|^2$ is a constant on $M$.
\end{lemma}

\begin{proof}
For simplicity, we denote $\psi:= \alpha^{-1}\varphi$. Then
\begin{align*}
\Delta_g u &= - \frac{n-2}{2} (1+ \psi)^{- \frac{n}{2}} \Delta_g \psi + \frac{n(n-2)}{4}  (1+ \psi)^{- \frac{n + 2}{2} } |\nabla \psi|^2 \\
&= \frac{n(n-2)}{4} (1+ \psi)^{- \frac{n+2}{2}} \left( 2 (1+ \psi) \psi +  |\nabla \psi|^2 \right).
\end{align*}
Now from the conformal transformation law for scalar curvature,
\begin{align*}
R_{\hat{g}} =&u^{-\frac{n+2}{n-2}}\left(R_gu-\frac{4(n-1)}{n-2}\Delta_g u\right) \\
=& R_g(1+\psi)^2-n(n-1) \left[2(1+\psi)\psi+|\nabla\psi|^2 \right]\\
\geq&n(n-1) \left[(1+\psi)^2-2(1+\psi)\psi-|\nabla\psi|^2 \right]\\
= & n(n-1) \left( 1-\psi^2-|\nabla\psi|^2 \right)\\
=&n(n-1) \left[ 1- \alpha^{-2}\left(\varphi^2+|\nabla\varphi|^2 \right) \right] \\
\geq&0.
\end{align*}
The equality case follows trivially.
\end{proof}

As for the mean curvature, we have
\begin{lemma}\label{lem:thm_A_mean_curvature}
Suppose the boundary $(\Sigma, g|_{\Sigma})$ is $\mathscr{L}_n$-regular, then the mean curvature of $\Sigma_i$ for the conformal metric $\hat g$ is strictly positive and given by
$$H^i_{\hat{g}} = \left( 1 +  \alpha^{-1} w_i \right) H^i_g - (n-1) \alpha^{-1} \partial_{\nu_g}\varphi .$$
\end{lemma}

\begin{proof}
By the conformal transformation law of mean curvature, we obtain
\begin{align*}
H^i_{\hat{g}}
=u^{-\frac{2}{n-2}}\left(H^i_g+\frac{2(n-1)}{n-2}\partial_{\nu_g}\log u \right)
=\alpha^{-1} \left( \alpha +  w_i \right) \left(H^i_g-(n-1)(\alpha+w_i)^{-1} \partial_{\nu_g}\varphi \right)
> 0,
\end{align*}
on $\Sigma_i$, where the last inequality follows from the assumption that $(\Sigma, g|_\Sigma)$ is $\mathscr{L}_n$-regular.
\end{proof}

Now we give the proof of our main theorem:
\begin{proof}[Proof of Theorem \ref{thm:L_n-inequality}]
From the assumption that the boundary $(\Sigma, g|_\Sigma)$ is $\mathscr{L}_n^w$-regular, by definition, each connected component $\Sigma_i$ with the conformal metric $\hat g|_{\Sigma_i}$ can be embedded in $\mathbb{R}^n$ as a convex hypersurface. In particular, $\Sigma_i$ has positive mean curvature $\hat H_0^i$ in $\mathbb{R}^n$. Now with the aid of Lemma \ref{lem:conformal_scalar_curvature} and \ref{lem:thm_A_mean_curvature}, we have
\begin{align*}
m_{BY}(\Sigma_i,\hat{g})
=\int_{\Sigma_i}\left(\hat H_0^i - H_{\hat{g}}^i\right)d\sigma_{\hat{g}}
=\int_{\Sigma_i}\left[ \hat H_0^i - \rho_i^{-1} H^i_g + (n-1) \alpha^{-1} \partial_{\nu_g}\varphi \right]\rho_i^{n-1}d\sigma_g
\geq0,
\end{align*}
due to the \emph{Positive Mass Theorem for Brown-York mass} (see \cite{Shi-Tam, Schoen-Yau_3}). That is,
\begin{align*}
 \int_{\Sigma_i}H^i_g \left(1+\frac{w_i}{\alpha}\right)^{2-n}d\sigma_g \leq \int_{\Sigma_i} \left(\hat H_0^i +(n-1)\frac{\partial_{\nu_g} \varphi}{\alpha} \right) \left(1+\frac{w_i}{\alpha} \right)^{1-n}d\sigma_g.
 \end{align*}
 
The rigidity of the equality case follows a similar argument in \cite{Yuan_2}. For the convenience of readers, we include it here. It would be interesting to compare both of them.

If $(M^n, g)$ is a convex domain in $\mathbb{S}^n$ and $\varphi$ is the restriction of a first eigenfunction of $\Delta_{\mathbb{S}^n}$, then $\hat g$ is flat and the Brown-York mass $m_{BY}(\Sigma, \hat g)$ vanishes. Therefore, the equality holds in this case.

Conversely, suppose the equality holds for some component $\Sigma_{i_0}$, it is equivalent that the Brown-York mass
$$m_{BY}(\Sigma_{i_0},\hat{g}) = 0,$$
which means the conformal metric $\hat g$ is flat and $(M, \hat g)$ is a connected compact convex domain in $\mathbb{R}^n$ due to the rigidity of vanishing Brown-York mass.
This implies that $g$ is conformally flat with scalar curvature $R_g = n(n-1)$ and $$\varphi^2 + |\nabla \varphi|^2 = \alpha^2$$ is a constant on $M$ by Lemma \ref{lem:conformal_scalar_curvature}. 

In order to show $g$ is spherical, we only need to show it is also Einstein. The essential idea is an adapted version of the proof for \emph{Obata theorem} (see Proposition 3.1 in \cite{Lee-Parker}).

For simplicity, we denote 
$$v:= \rho^{-1} = 1 + \alpha^{-1} \varphi$$
and hence $\hat g = v^{-2} g$. 
This gives us 
$$E_{\hat g} = E_g + (n-2) v^{-1} \left( \nabla^2 v - \frac{1}{n} g \Delta_g v \right) = 0.$$
That is,
$$E_g = - (n-2) v^{-1} \left( \nabla^2 v - \frac{1}{n} g \Delta_g v \right).$$
Thus, we have
\begin{align*}
\int_M v |E_g|^2 dv_g =& - (n-2) \int_M \langle \nabla^2 v - \frac{1}{n} g \Delta_g v , E_g \rangle dv_g \\
=& (n-2) \int_M \langle \nabla v, div_g E_g \rangle dv_g - (n-2) \int_\Sigma E_g (\nabla v, \nu) d\sigma_g \\
=& - (n-2) \int_\Sigma E_g (\nabla v, \nu) d\sigma_g,
\end{align*}
where $$div_g E_g = \frac{n-2}{2n} dR_g = 0$$
by the \emph{contracted second Bianchi identity}.

On the other hand, 
\begin{align*}
E_g (\nabla v, \nu) =& - (n-2) v^{-1} \left( \nabla^2 v (\nabla v, \nu)  - \frac{1}{n} g(\nabla v, \nu)  \Delta_g v \right) \\
=& - (n-2)  \alpha^{-2}\left( - |\nabla \varphi|_g^{-1} \nabla^2 \varphi (\nabla \varphi , \nabla \varphi)  + \frac{1}{n} |\nabla \varphi|_g  \Delta_g \varphi \right) \\
=&  (n-2) \alpha^{-2} \left( |\nabla \varphi|_g^{-1} \nabla^2 \varphi (\nabla \varphi , \nabla \varphi)  + \varphi |\nabla \varphi|_g \right) \\
=& - \frac{n-2}{2} \alpha^{-2} \nabla_\nu |\nabla \varphi|^2 + (n-2) \alpha^{-2} w |\nabla \varphi|_g
\end{align*}
on $\Sigma$. Since $$\varphi^2 + |\nabla \varphi|^2  = \alpha^{-2}$$ on $M$, we have 
$$\nabla_\nu |\nabla \varphi|^2 = - \nabla_\nu \varphi^2 = - 2 \varphi \nabla_\nu \varphi = - 2 w |\nabla \varphi|_g.$$
Thus
$$E_g (\nabla v, \nu) = 0$$
and hence
$$\int_M v |E_g|^2 dv_g  =  - (n-2) \int_\Sigma E_g (\nabla v, \nu) d\sigma_g = 0.$$
\emph{i.e.} $(M^n, g)$ is Einstein. Together with the fact that $(M^n, g)$ is conformally flat and scalar curvature $R_g = n(n-1)$, we conclude that $(M^n, g)$ is isometric to a connected compact domain in the round sphere $\mathbb{S}^n$. Clearly, $(M^n, g)$ is convex when embedded in $\mathbb{S}^n$, since $(M^n, \hat g)$ is convex in $\mathbb{R}^n$ and conformal to $(M^n, g)$. Easy to see $\varphi$ is the restriction of a first eigenfunction from the equation it satisfies.
\end{proof}

For the purposes of this article, we are more interested in the case when $w$ is a constant. For conveniences of our applications, we summarize it in the following corollary. Note that we can get an analogous a priori estimate for first eigenfunction derived in \cite{Yuan_2} if we simply take the parameter $\tau$ to be zero. 
\begin{corollary}\label{cor:estimate_tau}
Let $(M^n,g)$ be an $n$-dimensional compact Riemannian manifold with boundary $\Sigma=\bigcup^m_{i=1}\Sigma_i $. Suppose its scalar curvature satisfies $R_g\geq n(n-1)$ and the boundary $(\Sigma, g|_\Sigma)$ can be embedded in $\mathbb{R}^n$ as a convex hypersurface. 

For a given constant $\tau \geq 0$, suppose $\varphi > 0$ solves the boundary value problem
\begin{equation}
    \left\{
\begin{split}
\Delta_g \varphi + n \varphi&=0,   \quad &\text{on} \  M \\
 \varphi&=\tau, \quad &\text{on} \  \Sigma
 \end{split}
\right.
\end{equation}
and the mean curvature of $\Sigma$ with respect to the metric $g$ satisfies that
\begin{align}
H_g^i > (n-1)(\alpha+ \tau)^{-1} \partial_{\nu_g} \varphi
\end{align}
on $\Sigma_i$, $i = 1, \cdots, m,$ where $$\alpha:=\max_{\overline{M}}\left(\varphi^2+|\nabla\varphi|^2 \right)^{\frac{1}{2}}.$$ Then the following inequalities hold
  \begin{equation}\label{eqn:generalized_eigenfct_est}
\int_{\Sigma_i}|\nabla \varphi|d\sigma_g \leq \frac{m_{BY}(\Sigma_i, g)}{n-1} \left(\max_{\overline{M}}\left(\varphi^2+|\nabla\varphi|^2 \right)^{\frac{1}{2}} + \tau \right) , \quad i=1,\cdots, m,
 \end{equation}
where the equality holds for some component $\Sigma_{i_0}$ if and only if $(M^n,g)$ is isometric to a geodesic ball in the round sphere $\mathbb{S}^n$.
\end{corollary}

\begin{proof}
It is easy to see that equation (\ref{eqn:generalized_eigenfct_est}) can be easily deduced from Theorem \ref{thm:L_n-inequality} simply by taking $w = \tau$. Note that in the setting here, we have $\partial_{\nu_g} \varphi < 0$ by \emph{Hopf's lemma}. For the equality case, $(M^n,g)$ has to be a convex compact domain in $\mathbb{S}^n$ by the corresponding part of Theorem \ref{thm:L_n-inequality}. However, since $(\Sigma, g|_\Sigma)$ can be embedded in $\mathbb{R}^n$, it has to be the intersection of $\mathbb{S}^n$ with a hyperplane in $\mathbb{R}^{n+1}$ and hence is the boundary of some geodesic ball of $\mathbb{S}^n$. This implies $(M^n,g)$ is a geodesic ball in $\mathbb{S}^n$.
\end{proof}

\ \\


\section{Applications on CPE metrics}

For $n\geq3$, let $(M^n, g)$ be an $n$-dimensional closed manifold with CPE metric. We will apply Corollary \ref{cor:estimate_tau} to derive an estimate for the area of each component of critical level set $\Sigma$. \\

Denote 
$$M_+ := \{ x \in M: u(x) \geq -1\}, \qquad M_- := \{ x \in M: u(x) \leq -1\}$$
and
$$\Sigma := M_+ \cap M_- =  \{ x \in M: u(x) = -1\}.$$
Without losing of generality we can assume $M_- \backslash \Sigma \neq \emptyset$, otherwise $u \geq -1$ and Besse's conjecture holds according to Hwang's result (see \cite{Hwang_1}).\\

A fundamental feature of CPE metrics is that they satisfies the following Robinson-type identity:
\begin{lemma}\label{lem:Robinson_identity}
	Suppose $g$ is CPE metric with scalar curvature $R_g=n(n-1)$, then
	$$
	\Delta_g (|\nabla u|^2+ u^2)- \left(u + 1 \right)^{-1}\nabla u \cdot\nabla(|\nabla u|^2+ u^2)
	=2|\nabla_g^2 u + u g|^2,
	$$
	on $M \backslash \Sigma$.
\end{lemma}

\begin{proof}
	By Bochner-Weitzenb\"ock formula, 
	\begin{align*}
	\frac{1}{2} \Delta_g \left(|\nabla u|^2 + u^2 \right) &= |\nabla_g^2 u|^2 + \nabla u \cdot \nabla \Delta_g u + Ric_g(\nabla u, \nabla u) +  u \Delta_g u + |\nabla u|^2  \\
	&= |\nabla_g^2 u+ u g|^2 - u \Delta_g u - n u^2  - (n-1) |\nabla u|^2 + Ric_g(\nabla u, \nabla u) \\
	&= |\nabla_g^2 u+ u g|^2 - (n-1) |\nabla u|^2 + Ric_g(\nabla u, \nabla u) .
	\end{align*}
	On the other hand, from equation (\ref{eqn:V-static_phi}), 
	\begin{align*}
	&  Ric_g (\nabla u,\nabla u) - (n-1) |\nabla u|^2 \\
	=& \nabla_g^2 u (\nabla u, \nabla u)  - |\nabla u|^2 \Delta_g u - u Ric_g (\nabla u,\nabla u) \\
	=& \frac{1}{2} \nabla u \cdot \nabla |\nabla u|^2 + n u |\nabla u|^2 - u Ric_g (\nabla u,\nabla u) 
	\end{align*}
	That is,
	\begin{align*}
	&\left(u + 1\right)Ric(\nabla u,\nabla u)=(n u + (n-1))|\nabla u|^2 + \frac{1}{2} \nabla u \cdot \nabla |\nabla u|^2.
	\end{align*}
	Now we have
	\begin{align*}
	&\Delta_g (|\nabla u|^2+ u^2) \\
	=&2|\nabla_g^2 u+ u g|^2-2(n-1)|\nabla u|^2 +2 Ric_g (\nabla u,\nabla u)\\
	=&2|\nabla^2 u+ u g|^2 -2(n-1) |\nabla u|^2 +2 \left( u + 1\right)^{-1} \left( (n u + (n-1))|\nabla u|^2 + \frac{1}{2} \nabla u \cdot \nabla |\nabla u|^2 \right)\\
	=&2|\nabla^2 u+ u g|^2+ \left( u + 1\right)^{-1} \nabla u\cdot \nabla(|\nabla u|^2 + u^2).
	\end{align*}
\end{proof}

Immediately, we have
\begin{proposition}\label{prop:CPE-maximum_principle}
	\begin{align}
	\max_{M_-} \left( |\nabla u|^2+ u^2 \right) = \max_\Sigma \left( |\nabla u|^2+ u^2 \right)
	\end{align}
\end{proposition}

\begin{proof}
	For any $\delta >0$, let $M_\delta:= \{ x\in M: u (x) \leq -1 - \delta \}$. 
	Applying \emph{maximum principle}, we conclude that
	$$\max_{M_\delta} \left( |\nabla u|^2+ u^2 \right) = \max_{\partial M_\delta} \left( |\nabla u|^2+ u^2 \right).$$
	Now the conclusion follows, since $M_- = \bigcup_{\delta > 0} M_\delta$.
\end{proof}

\begin{remark}
	It seems that we can conclude that $|\nabla u|^2+ u^2$ is a constant on $M$ from Lemma \ref{lem:Robinson_identity} and maximum principle, since its maximum can be achieved on $M$. However, due to the existence of the factor $(1+u)^{-1}$ before the first order term, which goes to infinity when it approaches to $\Sigma$, the equation is degenerate at $\Sigma$ and won't be elliptic any more. Thus, we can not achieve the constancy of $|\nabla u|^2+ u^2$ in such a simple manner. Nevertheless, Proposition \ref{prop:CPE-maximum_principle} still holds by continuity.\\
\end{remark}

Now we give a general area estimates for components of critical level set $\Sigma$:

\begin{proof}[Proof of Theorem A]
	Let $\varphi = -u$ and take the trace of CPE equations (\ref{eqn:CPE}) with normalization $R_g = n(n-1)$, we have
	\begin{equation}\label{eqn:V-static_phi_trace}
	\left\{
	\begin{split}
	\mathscr{L}_n\varphi&=0,   \quad &\text{on} \  M_- \\
	\varphi &=  1, \quad &\text{\ on} \  \Sigma.\ \  
	\end{split}
	\right.
	\end{equation}
	According to Definition \ref{def:L_lambda_extension}, $\varphi$ is a positive $\mathscr{L}_n$-extension of $\tau = 1$ on $\Sigma$.
		
	Now we characterize the extrinsic geometry of the critical level set $\Sigma$. By evaluating $\varphi$ to be $1$, we have
	\begin{align}\label{eqn:CPE_-1_level_set}
	\nabla^2 \varphi = - g
	\end{align}
	holds on $\Sigma$. From this, we can derive an very important feature of $\Sigma$. That is, for each $i = 1, \cdots, m$, the quantity $W_i:= |\nabla u|_{\Sigma_i} = |\nabla \varphi|_{\Sigma_i}$ is a strictly positive constant on the component $\Sigma_i$, unless $\Sigma_i$ is a single point. To see this, we can choose an orthonormal frame 
	$$\{e_1,e_2,\cdots,e_{n-1},e_n\}$$ 
	along $\Sigma_i$, where $e_n=\nu :=-\frac{\nabla\varphi}{|\nabla\varphi|}$ is the outward unit normal vector field on $\Sigma_i$. For any $k = 1, \cdots, n-1$, we have $$e_k (W_i^2) = \nabla_k |\nabla \varphi|^2 = 2 \nabla^2 \varphi ( \nabla \varphi, e_k) = - 2|\nabla \varphi|_{\Sigma_i} \varphi_{n k} = 2 W_i  g_{n k} = 0,$$
	which shows that $W_i$ is a nonnegative constant on the boundary $\Sigma$. On the other hand, suppose $W_i$ vanishes on $\Sigma_i$, in particular at some point $p\in \Sigma_i$. According to equation (\ref{eqn:CPE_-1_level_set}), $p$ would be a non-degenerate critical point of $\varphi$ and hence $\Sigma_i = \{p\}$. 
	
	When $\Sigma_i$ is a single point, the area estimate holds automatically and hence without losing of generality we may assume $\Sigma_i$ has positive $(n-1)$-dimensional measure and $W_i \neq 0$ in the rest of the argument.
	
	From equation (\ref{eqn:CPE_-1_level_set}), we can also read the mean curvature of  $\Sigma_i$ very quickly:
	$$H^i_g = - W_i^{-1} \sum_{j = 1}^{n-1} \varphi_{jj}= (n-1)W_i^{-1},$$
	which shows all components of $\Sigma$ are in fact hypersurfaces of constant mean curvature.
	By Proposition \ref{prop:CPE-maximum_principle},
		$$\alpha := \max_{M_-} ( \varphi^2 + |\nabla \varphi|^2 )^{\frac{1}{2}} = \max_\Sigma ( \varphi^2 + |\nabla \varphi|^2 )^{\frac{1}{2}} = \left( 1 + \max_i W_i^2 \right)^{\frac{1}{2}} > 1, $$
	then we have
	$$H_g^i > 0 > - (n-1) (\alpha + 1)^{-1} W_i.$$
	
	Now applying Corollary \ref{cor:estimate_tau}, we have
	$$
	Area (\Sigma_i, g)\leq \frac{\alpha + 1}{(n-1) W_i}\int_{\Sigma_i} \left( H_0^i - H_g^i \right) d\sigma_g .
	$$
	That is,
	$$
	Area (\Sigma_i, g)\leq \frac{(\alpha +1) W_i }{(n-1) (W_i^2 + \alpha + 1) } \int_{\Sigma_i} H_0^i d\sigma_g.
	$$
	On the other hand, we have Gauss equation
	$$R_\Sigma = H_0^2 - |A_0|^2 = \frac{n-2}{n-1}H_0^2 - |\overset{\circ}{A}{}_0|^2, $$
	since $\Sigma$ is embedded in $\mathbb{R}^n$. Now by H\"older inequality, we have
	\begin{align*}
	Area (\Sigma_i, g)\leq& \frac{ (\alpha + 1)^2 W_i^2 }{(n-1)^2 (W_i^2 + 1 + \alpha)^2 } \int_{\Sigma} H_0^2 d\sigma_g \\
	=& \frac{ (\alpha + 1)^2 W_i^2 }{(n-1)(n-2) (W_i^2 + 1 + \alpha)^2 } \int_{\Sigma} \left( R_\Sigma + |\overset{\circ}{A}{}_0|^2 \right) d\sigma_g.
	\end{align*}
\end{proof}

As a corollary of Theorem A, immediately we get
	\begin{equation}
	\text{Area}(\Sigma, g)\leq \frac{W^2}{(n-1)(n-2)( 1 + W^2)} \int_\Sigma \left( R_\Sigma + |\overset{\circ}{A}{}_0|^2 \right)d\sigma_g,
	\end{equation}
if $\Sigma$ is assumed to be connected. In fact, this estimate can be improved if we apply a similar \emph{Robinson-identity type} argument of Shen in \cite{Shen}. Before showing this, we present the following fundamental identity associated to a CPE manifold and the improved estimate is simply an application of this identity.\\

\begin{lemma}\label{lem:CPE_area_identity}
	Suppose $(M^n, g, u)$ is a CPE manifold with scalar curvature $R_g = n(n-1)$ and $\Sigma := \bigcup_{i=1}^m \Sigma_i$, then the following identity holds
	\begin{align}
		(n-1)(n-2) \sum_{i=1}^m \frac{1+W_i^2}{W_i} Area(\Sigma_i, g) = \sum_{i=1}^m W_i \int_{\Sigma_i} R_{\Sigma_i} d\sigma_g  - 2\int_{M_+}(1+ u)|E_g|^2 dv_g.
	\end{align}
\end{lemma}

\begin{proof}
	It is straightforward that from Lemma \ref{lem:Robinson_identity} we have the \emph{Robinson's identity}:
	\begin{align}\label{eqn:div}
	div_g \left( \frac{1}{1+ u} \nabla(|\nabla u|^2 + u^2)\right) = 2(1+ u)|E_g|^2.
	\end{align}
	For $\delta > 0$, taking integration on both sides of equation (\ref{eqn:div}) over $$M_\delta:= \{ x\in M: u(x) > -1 + \delta\}$$ and from integration by parts, we have
	\begin{align*}
	\int_{M_\delta}(1+ u)|E_g|^2 dv_g
	=&\frac{1}{2}\int_{M_\delta}div_g \left( \frac{1}{1+ u} \nabla(|\nabla u|^2 + u^2)\right) dv_g \\
	=&\frac{1}{2}\int_{\Sigma_\delta} \frac{1}{1+ u} \nabla_{\nu_g} (|\nabla u|^2 + u^2) d\sigma_g \\
	=& - \frac{1}{\delta} \int_{\Sigma_\delta} |\nabla u| \left(\nabla_{\nu_g}\nabla_{\nu_g} u + u \right) d\sigma_g \\
	=& -  \int_{\Sigma_\delta} |\nabla u| \left(Ric( \nu_g, \nu_g) - (n-1) \right) d\sigma_g,
	\end{align*}
	where $\Sigma_\delta = \partial M_\delta$ and $\nu_g$ is the outward normal of $\Sigma_\delta$. 
	Let $\delta \rightarrow 0$, we have
	\begin{align}\label{eqn:Sigma_alpha}
	\sum_i W_i \int_{\Sigma_i} \left(Ric( \nu_g, \nu_g) - (n-1) \right) d\sigma_g = - 	\int_{M_+}(1+ u)|E_g|^2 dv_g,
	\end{align}
	From Gauss equation,
	\begin{align*}
	R_\Sigma =& R_g - 2 Ric( \nu_g, \nu_g) + H_g^2 - |A_g|^2 \\
	=& n(n-1) - 2 Ric( \nu_g, \nu_g) + (n-1)(n-2) W^{-2},
	\end{align*}	
	where the second fundamental form and mean curvature of $\Sigma$ is given by $A_g = - W^{-1} g_\Sigma$ and $H_g = - (n-1)W^{-1}$. Thus equality (\ref{eqn:Sigma_alpha}) can be written as
	\begin{align*}
	(n-1)(n-2) \sum_i \frac{1+W_i^2}{W_i} Area(\Sigma_i, g) = \sum_i W_i \int_{\Sigma_i} R_{\Sigma_i} d\sigma_g  - 2\int_{M_+}(1+ u)|E_g|^2 dv_g.
	\end{align*}
	\end{proof}

Now we have
\begin{proof}[Proof of Theorem B]
	 From Lemma \ref{lem:CPE_area_identity}, when $\Sigma$ is connected, we have
	\begin{align*}
	 Area(\Sigma, g) =& \frac{W^2}{(n-1)(n-2)(1+W^2)} \int_{\Sigma} R_{\Sigma} d\sigma_g  - \frac{2W}{(n-1)(n-2)(1+W^2)}\int_{M_+}(1+ u)|E_g|^2 dv_g\\
	 \leq& \frac{W^2}{(n-1)(n-2)(1+W^2)} \int_{\Sigma} R_{\Sigma} d\sigma_g,
	\end{align*}		
	where equality holds if and only if $E_g = 0$ on $M_+$. On the other hand, we have
	\begin{align}
	\int_M (1+ u)|E_g|^2 dv_g = 0,
	\end{align}
	from equality (\ref{eqn:div}). This shows that $E_g$ vanishes on $M$ if and only if it vanishes on $M_+$ and it concludes the theorem.

	Now for $n=3$ and applying Lemma \ref{lem:CPE_area_identity} and \emph{Gauss-Bonnet formula}, we get
	\begin{align}
	 \sum_i \frac{1+W_i^2}{W_i} Area(\Sigma_i, g) + \int_{M_+}(1+ u)|E_g|^2 dv_g = 2\pi\sum_i  W_i \chi(\Sigma_i).
	\end{align}	
	In particular, there exists at least a component $\Sigma_{i_0}$ admits positive Euler characteristic and hence is homeomorphic to $\mathbb{S}^2$ since it is orientable. Suppose $\Sigma$ is connected, from the argument above, we can conclude that $\Sigma$ is homeomorphic to $\mathbb{S}^2$ with 	
	\begin{align}
	Area(\Sigma, g)  =  \frac{4 \pi W^2}{1+W^2} - \frac{W}{1+W^2} \int_{M_+}(1+ u)|E_g|^2 dv_g \leq \frac{4 \pi W^2}{1+W^2},
	\end{align}		
	where equality holds if and only if $E_g$ vanishes on $M$.
\end{proof}

\begin{remark}
	By comparing proofs of Theorem A and B, we can find an interesting fact that the integral term
	$$\frac{2W}{(n-1)(n-2)(1+W^2)}\int_{M_+}(1+ u)|E_g|^2 dv_g$$
	represents the Brown-York mass $m_{BY}(\Sigma, \hat g)$. This fact might be helpful for us to have a better understanding of Besse's conjecture and $V$-static metrics.
\end{remark}

\ \\

\section{Area estimates for the boundary of $V$-static manifold}

For $n\geq3$, let $(M^n, g)$ be an $n$-dimensional compact $V$-static manifold with boundary $\Sigma:= \partial M$. The area estimate for each component $\Sigma_i$ is in fact equivalent to those for the critical level set in a CPE manifold. Hence it is straightforward to get these estimates simply by applying Theorem A:

\begin{proof}[Proof of Corollary A]
Recall the definition of $V$-static metric, we have
\begin{equation}\label{eqn:V-static}
    \left\{
\begin{split}
\gamma_g^* f &= \mu g,   \quad &\text{on} \  M \\
f&=0, \quad &\text{on} \  \Sigma,
 \end{split}
\right.
\end{equation}
where $\gamma_g^* f = \nabla_g^2f-g\Delta_g f-fRic_g$ and $f > 0$ on $M \backslash \Sigma$. \\

For the case $\mu = 0$, it has been solved by the second named author in \cite{Yuan_2} and hence we only need to focus on the case when $\mu > 0$.\\

For normalized scalar curvature $R_g = n(n-1)$, let $u := - \frac{n-1}{\mu} f - 1$, we can rewrite the $V$-static equation as follow:
\begin{equation}\label{eqn:V-static_phi}
    \left\{
\begin{split}
\gamma_g^* u &= E_g,   \quad &\text{on} \  M \\
u &= - 1, \quad &\text{on} \  \Sigma,
 \end{split}
\right.
\end{equation}
where $u < -1$ on $M \backslash \Sigma$. This is exactly the CPE equation on $M_-$ in the previous section. Therefore Theorem A can be applied and the conclusion follows.
\end{proof}

\ \\

\bibliographystyle{amsplain}

\end{document}